\documentclass[12pt]{article}%
\usepackage[pdftex]{graphicx}
\usepackage{epstopdf}
\usepackage{amsmath}
\usepackage{amsfonts}
\usepackage{amssymb}
\usepackage{graphicx}
\usepackage{hyperref}%
\providecommand{\U}[1]{\protect\rule{.1in}{.1in}}
\newtheorem{theorem}{Theorem}

\newtheorem{proposition}[theorem]{Proposition}
\newtheorem{remark}[theorem]{Remark}

\newenvironment{proof}[1][Proof]{\noindent\textbf{#1.} }{\ \rule{0.5em}{0.5em}}
\begin{document}

\title{Zero distribution of polynomials satisfying a differential-difference equation}
\author{Diego Dominici\thanks{Supported by a Humboldt Research Fellowship for
experienced researchers of the Alexander von Humboldt Foundation}\quad\ and
\quad Walter Van Assche\thanks{Supported by KU Leuven Research Project
OT/12/073, FWO Project G.0934.13 and the Belgian Interuniversity Attraction
Poles Programme P7/18.}\\{\small State University of New York at New Paltz, USA }\\{\small KU Leuven, Belgium}}

\maketitle

\begin{quotation}
Dedicated to Frank Olver, who showed us all the asymptotic path.
\end{quotation}

\begin{abstract}
In this paper we investigate the asymptotic distribution of the zeros of
polynomials $P_{n}(x)$ satisfying a first order differential-difference
equation. We give several examples of orthogonal and non-orthogonal families.

\end{abstract}

MSC-class: 34E05 (Primary) 11B83, 33C45, 44A15 (Secondary)

Keywords: Differential-difference equations, polynomial sequences, Stieltjes transform, zero counting measure

\section{Introduction}

Many families of polynomials $P_{n}(x)$ satisfy differential-difference
equations of the form
\begin{equation}
P_{n+1}(x)=A_{n}(x)P_{n}^{\prime}(x)+B_{n}(x)P_{n}(x),\qquad n\geq0,
\label{eq:diff2}%
\end{equation}
where $P_{0}(x)=1$, and $A_{n}(x)$, $B_{n}(x)$ are polynomials of degree at
most $2$ and $1$ respectively.

When $A_{n}(x)$ and $B_{n}(x)$ are independent of $n,$ we can identify
$P_{n}(x)$ with some class of \emph{derivative polynomials} \cite{MR1321452}
defined by%
\[
P_{n+1}(x)=Q(x)P_{n}^{\prime}(x)+xP_{n}(x),
\]
where the polynomial $Q(x)$ satisfy%
\[
f^{\prime}(x)=Q\left[  f(x)\right]
\]
for some function $f(x).$ These polynomials have the pseudo-Rodrigues formula%
\[
P_{n}\left[  f(x)\right]  =\frac{1}{g(x)}\frac{d^{n}}{dx^{n}}g(x),
\]
with%
\[
g(x)=\exp\left(
%TCIMACRO{\dint \limits^{x}}%
%BeginExpansion
{\displaystyle\int\limits^{x}}
%EndExpansion
f(t)dt\right)  .
\]

Examples of derivative polynomials include the monic Hermite polynomials
$\widehat{H}_{n}(x)$, defined by \cite{MR2656096}%
\[
\widehat{H}_{n+1}(x)=-\frac{1}{2}\widehat{H}_{n}^{\prime}(x)+x\widehat{H}%
_{n}(x).
\]
In this case,%
\[
Q(x)=-\frac{1}{2},\quad f(x)=-\frac{1}{2}x,\quad g(x)=\exp\left(  -\frac
{x^{2}}{4}\right)  .
\]
We will analyze the Hermite polynomials in Section \ref{Hermite}.

Another example comes from taking
\[
Q(x)=x,\quad f(x)=e^{x},\quad g(x)=\exp\left(  e^{x}\right)  .
\]
In this case, the polynomials $P_{n}(x)$ satisfy%
\[
P_{n+1}(x)=xP_{n}^{\prime}(x)+xP_{n}(x),
\]
and are called Bell polynomials \cite{MR1503161}. We analyze these polynomials
in Section \ref{Bell}.

Now suppose that $H(x)=h^{-1}(x),$ the inverse function of $h(x)$, and that
\[
f(x)=\frac{1}{h^{\prime}(x)},\quad z_{0}=h(x_{0}),\text{ \ }\ \left\vert
f(x_{0})\right\vert \in\left(  0,\infty\right)  .
\]
Then \cite{MR2427672},$\ $
\[
\frac{d^{n}H}{dz^{n}}(z_{0})=\left[  f(x_{0})\right]  ^{n}g_{n-1}(x_{0}),\quad
n=1,2,\ldots,
\]
where the functions $g_{n}(x)$ satisfy $g_{0}(x)=1$ and
\[
g_{n+1}(x)=g_{n}^{\prime}\left(  x\right)  +\left(  n+1\right)  \frac
{f^{\prime}(x)}{f(x)}g_{n}\left(  x\right)  ,\quad n=0,1,\ldots.
\]
If $f(x)=\exp\left(  ax^{2}+bx\right)  ,$ the functions $g_{n}(x)$ are
polynomials of degree $n.$ In particular, when $a=\frac{1}{2},$ $b=0,$ we
obtain a family of polynomials associated with the derivatives of the inverse
error function. We analyze these polynomials in Section \ref{Error}.

Polynomial solutions of (\ref{eq:diff2}) arise naturally in combinatorics as
generating functions of sequences of numbers having a combinatorial
interpretation. For example, the Bell polynomials are generating functions for
the Stirling numbers of the second kind \cite{MR2172781}
\[
P_{n}(x)=%
%TCIMACRO{\dsum \limits_{k=0}^{n}}%
%BeginExpansion
{\displaystyle\sum\limits_{k=0}^{n}}
%EndExpansion
\left\{
\begin{array}
[c]{c}%
n\\
k
\end{array}
\right\}  x^{k},
\]
where $\left\{
\begin{array}
[c]{c}%
n\\
k
\end{array}
\right\}  $ represents the number of ways to partition a set of $n$ objects
into $k$ non-empty subsets.

The location of the zeros of the generating function $G(x)$ of a sequence
$c_{n}$ determines some of the properties of $c_{n}.$ For example, when $G(x)$
is a polynomial, we have the following result \cite{MR1110850}:

\begin{theorem}
Let%
\[
p(x)=c_{0}+c_{1}x+\cdots+c_{n}x^{n}%
\]
be a polynomial all of whose zeros are real and negative. Then, the
coefficient sequence $c_{n}$ is strictly log concave.
\end{theorem}

An extension of this result was proven by Schoenberg \cite{MR0072918}.

In this paper, we will analyze the asymptotic distribution of the zeros of
polynomials defined by the differential-difference equation (\ref{eq:diff2}).

\section{Interlacing zeros}

\strut In \cite{MR2741218}, we studied polynomial solutions of (\ref{eq:diff2}%
). Under some mild conditions on the coefficients $A_{n}(x)$, $B_{n}(x),$ we
concluded that in general the zeros of the polynomials $P_{n}(x)$ are real and
interlace. This result would be trivial if the polynomials $P_{n}(x)$ are
orthogonal but, in almost all cases, they are not.

The following theorem is crucial and of independent interest.

\begin{theorem}
\label{thm2} \label{thm1} Suppose $P_{n}$ and $P_{n+1}$ have interlacing zeros
for every $n\in\mathbb{N}$, i.e.,
\begin{equation}
x_{k,n}\leq x_{k,n-1}\leq x_{k+1,n},\qquad1\leq k\leq n-1,
\label{eq:interlace}%
\end{equation}
and that $x_{k,n}/\phi(n)\in\lbrack a,b]$ for every $1\leq k\leq n$ and
$n\in\mathbb{N}$, where $\phi(n)$ is a positive and increasing sequence. Then
there exists an infinite subset $\Lambda\subset\mathbb{N}$ such that
\begin{equation}
\lim_{n\rightarrow\infty,n\in\Lambda}\frac{\phi(n)}{n}\frac{P_{[nt]}^{\prime
}(\phi(n)x)}{P_{[nt]}(\phi(n)x)}=f_{\Lambda}(t,x) \label{limit}%
\end{equation}
for some function $f$ which is continuous on $[0,1]\times\mathbb{C}%
\setminus\lbrack a^{\ast},b^{\ast}]$, and the convergence is uniform for
$t\in\lbrack0,1]$ and $x\in\mathbb{C}\setminus\lbrack a^{\ast},b^{\ast}]$. The
points $a^{\ast},b^{\ast}$ are given by
\begin{equation}
a^{\ast}=\min\{0,a\},\quad b^{\ast}=\max\{0,b\}, \label{eq:abstar}%
\end{equation}
whenever $\phi(n)$ increases to $\infty$, and $a^{\ast}=a$, $b^{\ast}=b$
whenever $\phi$ is constant.
\end{theorem}

\begin{proof}
The partial fraction decomposition
\begin{equation}
\frac{P_{m}^{\prime}(x)}{P_{m}(x)}=\sum_{k=1}^{m}\frac{1}{x-x_{k,m}}
\label{Partial}%
\end{equation}
readily gives
\begin{align*}
\lefteqn{\phi(N)\left(  \frac{P^{\prime}_{m}(\phi(N)x)}{P_{m}(\phi(N)x)}%
-\frac{P^{\prime}_{m-1}(\phi(N)x)}{P_{m-1}(\phi(N)x)}\right)  }\\
&  =\sum_{k=1}^{m}\frac{1}{x-x_{k,m}/\phi(N)}-\sum_{k=1}^{m-1}\frac
{1}{x-x_{k,m-1}/\phi(N)}\\
&  =\frac{1}{x-x_{m,m}/\phi(N)}+\frac{1}{\phi(N)}\sum_{k=1}^{m-1}%
\frac{x_{k,m-1}-x_{k,m}}{(x-x_{k,m}/\phi(N))(x-x_{k,m-1}/\phi(N))}.
\end{align*}
If $N\geq m$ then
\[
\frac{x_{k,m}}{\phi(N)}=\frac{x_{k,m}}{\phi(m)}\frac{\phi(m)}{\phi(N)}%
\in\left[  \frac{\phi(m)}{\phi(N)}a,\frac{\phi(m)}{\phi(N)}b\right]  .
\]
Since $\phi$ is increasing, we have $\phi(m)/\phi(N)\leq1$, and if
$\phi(n)\rightarrow\infty$, then
\[
\lim_{N\rightarrow\infty}\phi(m)/\phi(N)=0
\]
for every fixed $m$. This means that $x_{k,m}/\phi(N)\in\lbrack a^{\ast
},b^{\ast}]$ for every $m\leq N$. Let $K$ be a compact set in $\mathbb{C}%
\setminus\lbrack a^{\ast},b^{\ast}]$ and let $\delta$ be the distance from $K$
to $[a^{\ast},b^{\ast}]$, then $\delta>0$ and
\[
|x-x_{k,m}/\phi(N)|\geq\delta,\qquad x\in K.
\]
Hence
\[
\phi(N)\left\vert \frac{P_{m}^{\prime}(\phi(N)x)}{P_{m}(\phi(N)x)}%
-\frac{P_{m-1}^{\prime}(\phi(N)x)}{P_{m-1}(\phi(N)x)}\right\vert \leq\frac
{1}{\delta}+\frac{1}{\delta^{2}}\sum_{k=1}^{m-1}\frac{x_{k,m-1}-x_{k,m}}%
{\Phi(N)}.
\]
The interlacing (\ref{eq:interlace}) implies that $x_{k,m-1}-x_{k,m}\geq0$,
and furthermore $x_{k,m-1}-x_{k,m}\geq x_{k+1,m}-x_{k,m}$, so that
\[
\sum_{k=1}^{m-1}\frac{x_{k,m-1}-x_{k,m}}{\Phi(N)}\leq\sum_{k=1}^{m-1}%
\frac{x_{k+1,m}-x_{k,m}}{\Phi(N)}=\frac{x_{m,m}-x_{1,m}}{\phi(N)}\leq b^{\ast
}-a^{\ast}%
\]
whenever $m\leq N$. This gives the bound
\[
\phi(N)\left\vert \frac{P_{m}^{\prime}(\phi(N)x)}{P_{m}(\phi(N)x)}%
-\frac{P_{m-1}^{\prime}(\phi(N)x)}{P_{m-1}(\phi(N)x)}\right\vert \leq\frac
{1}{\delta}+\frac{b^{\ast}-a^{\ast}}{\delta^{2}}%
\]
which holds for every $x\in K$ and every $m\leq N$. From this one easily
finds
\begin{align*}
\phi(N)\left\vert \frac{P_{m}^{\prime}(\phi(N)x)}{P_{m}(\phi(N)x)}%
-\frac{P_{\ell}^{\prime}(\phi(N)x)}{P_{\ell}(\phi(N)x)}\right\vert  &
\leq\phi(N)\sum_{j=\ell+1}^{m}\left\vert \frac{P_{j}^{\prime}(\phi(N)x)}%
{P_{j}(\phi(N)x)}-\frac{P_{j-1}^{\prime}(\phi(N)x)}{P_{j-1}(\phi
(N)x)}\right\vert \\
&  \leq\left(  \frac{1}{\delta}+\frac{b^{\ast}-a^{\ast}}{\delta^{2}}\right)
(m-\ell)
\end{align*}
whenever $\ell\leq m\leq N$. Now take $m=[nt],\ell=[ns]$, where $0\leq s\leq
t\leq1$, then
\[
\frac{\phi(N)}{n}\left\vert \frac{P_{[nt]}^{\prime}(\phi(N)x)}{P_{[nt]}%
(\phi(N)x)}-\frac{P_{[ns]}^{\prime}(\phi(N)x)}{P_{[ns]}(\phi(N)x)}\right\vert
\leq C\frac{[nt]-[ns]}{n}%
\]
holds for every $N\geq n$ and
\[
C=\frac{1}{\delta}+\frac{b^{\ast}-a^{\ast}}{\delta^{2}}.
\]
In particular we have for $N=n$ and
\[
f_{n}(t,x)=\frac{\phi(n)}{n}\frac{P_{[nt]}^{\prime}(\phi(n)x)}{P_{[nt]}%
(\phi(n)x)}%
\]
the inequality
\begin{equation}
|f_{n}(t,x)-f_{n}(s,x)|\leq C\frac{[nt]-[ns]}{n},\qquad n\in\mathbb{N},
\label{eq:equi}%
\end{equation}
whenever $0\leq s\leq t\leq1$. Let $D[0,1]$ be the space of functions
$f:[0,1]\rightarrow\mathbb{C}$ that are right-continuous and have left-hand
limits (see \cite[Chapter 3]{MR0233396}). In $D[0,1]$ we use the Skorohod
topology and the modulus of continuity
\[
w_{f}^{\prime}(\delta)=\inf_{\{t_{i}\}}\max w_{f}([t_{i-1},t_{i})),
\]
where the infimum is over all finite sets $\{t_{0},t_{1},\ldots,t_{r}\}$ of
points in $[0,1]$ satisfying $0=t_{0}<t_{i}<\cdots<t_{r}=1$ and $t_{i}%
-t_{i-1}>\delta$ for $1\leq i\leq r$, and
\[
w_{f}([t_{i-1},t_{i}))=\sup\{|f(s)-f(t)|:\ t_{i-1}\leq s,t<t_{i}\}.
\]
Observe that (\ref{eq:equi}) implies that
\[
w_{f_{n}}^{\prime}(\delta)=%
\begin{cases}
0 & 0\leq\delta<1/n,\\
C/n & 1/n\leq\delta<2/n,\\
2C/n & 2/n\leq\delta<3/n,\\
\ \vdots & \quad\vdots
\end{cases}
\]
so that $w_{f_{n}}^{\prime}(\delta)\leq C\delta$. The analogue of the
Arzel\`{a}-Ascoli theorem in $D[0,1]$ \cite[Thm.~14.3]{MR0233396} now implies
that the set $\{f_{n},n\in\mathbb{N}\}$ of functions in $D[0,1]$ has compact
closure, hence there exists a subsequence $(f_{n})_{n\in\Lambda}$ that
converges in the Skorohod topology to a function $f_{\Lambda}$ in $D[0,1]$.
The inequality (\ref{eq:equi}) implies that
\[
|f_{\Lambda}(t,x)-f_{\Lambda}(s,x)|\leq C(t-s),
\]
hence $f_{\Lambda}$ is continuous, and the convergence is in fact uniform on
$[0,1]$. Note that all inequalities hold uniformly for $x\in K$, hence
$f_{n}(t,x)$ is a normal family on $\mathbb{C}\setminus\lbrack a^{\ast
},b^{\ast}]$ for every $t$ and the convergence also holds uniformly for $x$ on
compact sets of $\mathbb{C}\setminus\lbrack a^{\ast},b^{\ast}]$.
\end{proof}

\section{Ratio asymptotics}

We will choose the positive increasing sequence $\phi(n)$ in such a way that
the limits
\begin{align}
\lim_{n\rightarrow\infty}\frac{n}{\phi^{2}(n)}A_{n}(\phi(n)x)  &
=a(x)\label{eq:Aa}\\
\lim_{n\rightarrow\infty}\frac{1}{\phi(n)}B_{n}(\phi(n)x)  &  =b(x)
\label{eq:Bb}%
\end{align}
exists. Note that $a$ and $b$ are polynomials of degree at most $2$ and $1$
respectively. Then the differential-difference equation (\ref{eq:diff2})
gives
\[
\frac{P_{[nt]+1}(\phi(n)x)}{P_{[nt]}(x)}=A_{[nt]}(\phi(n)x)\frac
{P_{[nt]}^{\prime}(\phi(n)x)}{P_{[nt]}(\phi(n)x))}+B_{[nt]}(\phi(n)x).
\]
We will assume that $\phi$ is regularly varying, i.e.,
\begin{equation}
\phi(n)=n^{\sigma}L(n),\qquad\sigma\geq0 \label{RV}%
\end{equation}
with
\[
\lim_{n\rightarrow\infty}\frac{L([nt])}{L(n)}=1,
\]
then Theorem~\ref{thm1} and (\ref{eq:Aa})--(\ref{eq:Bb}) imply that for the
subset $\Lambda\subset\mathbb{N}$ we have
\begin{equation}
\lim_{n\rightarrow\infty,n\in\Lambda}\frac{1}{\phi(n)}\frac{P_{[nt]+1}%
(\phi(n)x)}{P_{[nt]}(\phi(n)x)}=t^{2\sigma-1}a(xt^{-\sigma})f_{\Lambda
}(t,x)+t^{\sigma}b(xt^{-\sigma}), \label{eq:ratio}%
\end{equation}
which gives ratio asymptotics for the polynomials for the same subsequence
where Theorem~\ref{thm1} gave asymptotics.

\section{How to determine the zero distribution}

\begin{theorem}
\label{thm:main} Suppose that the polynomials $P_{n}$ satisfy the
differential-difference equation (\ref{eq:diff2}) and that
\begin{align}
\lim_{n\rightarrow\infty}\frac{n}{\phi^{2}(n)}A_{n}(\phi(n)x)  &
=a(x)\label{eq:Aa2}\\
\lim_{n\rightarrow\infty}\frac{1}{\phi(n)}B_{n}(\phi(n)x)  &  =b(x)
\label{eq:Bb2}%
\end{align}
where $\phi(n)=n^{\sigma}L(n)$ is regularly varying. Then the limit function
$f\left(  s,x\right)  $ in (\ref{limit}) is independent of the subsequence
$\Lambda$ and satisfies
\begin{equation}
f(s,x)=%
%TCIMACRO{\dint \limits_{0}^{s}}%
%BeginExpansion
{\displaystyle\int\limits_{0}^{s}}
%EndExpansion
\frac{\frac{d}{dx}\left[  t^{2\sigma-1}a(xt^{-\sigma})f(t,x)+t^{\sigma
}b(xt^{-\sigma})\right]  }{t^{2\sigma-1}a(xt^{-\sigma})f(t,x)+t^{\sigma
}b(xt^{-\sigma})}\,dt, \label{eq:f(s,x)}%
\end{equation}
with $f(0,x)=0$ and
\begin{equation}
\underset{x\rightarrow\infty}{\lim}xf(s,x)\rightarrow s. \label{largex}%
\end{equation}

\end{theorem}

\begin{proof}
First we use telescopic summation to write
\[
\frac{1}{n}\frac{P_{n}^{\prime}(x)}{P_{n}(x)}=\frac{1}{n}\sum_{k=0}%
^{n-1}\left(  \frac{P_{k+1}^{\prime}(x)}{P_{k+1}(x)}-\frac{P_{k}^{\prime}%
(x)}{P_{k}(x)}\right)  .
\]
Observe that
\begin{align*}
\frac{P_{k+1}^{\prime}(x)}{P_{k+1}(x)}-\frac{P_{k}^{\prime}(x)}{P_{k}(x)}  &
=\frac{P_{k+1}^{\prime}(x)P_{k}(x)-P_{k+1}(x)P_{k}^{\prime}(x)}{P_{k+1}%
(x)P_{k}(x)}\\
&  =\left(  \frac{P_{k+1}(x)}{P_{k}(x)}\right)  ^{\prime}/\left(
\frac{P_{k+1}(x)}{P_{k}(x)}\right)
\end{align*}
so that
\[
\frac{1}{n}\frac{P_{n}^{\prime}(x)}{P_{n}(x)}=\frac{1}{n}\sum_{k=0}%
^{n-1}\left(  \frac{P_{k+1}(x)}{P_{k}(x)}\right)  ^{\prime}/\left(
\frac{P_{k+1}(x)}{P_{k}(x)}\right)  .
\]
This can be written as an integral:
\[
\frac{1}{n}\frac{P_{n}^{\prime}(x)}{P_{n}(x)}=\int_{0}^{1}\left(
\frac{P_{[nt]+1}(x)}{P_{[nt]}(x)}\right)  ^{\prime}/\left(  \frac
{P_{[nt]+1}(x)}{P_{[nt]}(x)}\right)  \,dt.
\]
Now replace $x$ by $\phi(n)x$, then
\[
\frac{\phi(n)}{n}\frac{P_{n}^{\prime}(\phi(n)x)}{P_{n}(\phi(n)x)}=\phi
(n)\int_{0}^{1}\left(  \frac{P_{[nt]+1}(\phi(n)x)}{P_{[nt]}(\phi(n)x)}\right)
^{\prime}/\left(  \frac{P_{[nt]+1}(\phi(n)x)}{P_{[nt]}(\phi(n)x)}\right)
\,dt.
\]
If we now use Theorem \ref{thm1} and (\ref{eq:ratio}), then we find a
subsequence $\Lambda\subset\mathbb{N}$ such that
\[
f_{\Lambda}(1,x)=\int_{0}^{1}\frac{\frac{d}{dx}\left(  t^{2\sigma
-1}a(xt^{-\sigma})f_{\Lambda}(t,x)+t^{\sigma}b(xt^{-\sigma})\right)
}{t^{2\sigma-1}a(xt^{-\sigma})f_{\Lambda}(t,x)+t^{\sigma}b(xt^{-\sigma}%
)}\,dt.
\]
In the same way we can also find
\[
\frac{\phi(n)}{n}\frac{P_{[ns]}^{\prime}(\phi(n)x)}{P_{[ns]}(\phi(n)x)}%
=\phi(n)\int_{0}^{s}\left(  \frac{P_{[nt]+1}(\phi(n)x)}{P_{[nt]}(\phi
(n)x)}\right)  ^{\prime}/\left(  \frac{P_{[nt]+1}(\phi(n)x)}{P_{[nt]}%
(\phi(n)x)}\right)  \,dt.
\]
which for $n$ tending to infinity in $\Lambda$ gives
\[
f_{\Lambda}(s,x)=\int_{0}^{s}\frac{\frac{d}{dx}\left[  t^{2\sigma
-1}a(xt^{-\sigma})f_{\Lambda}(t,x)+t^{\sigma}b(xt^{-\sigma})\right]
}{t^{2\sigma-1}a(xt^{-\sigma})f_{\Lambda}(t,x)+t^{\sigma}b(xt^{-\sigma}%
)}\,dt,
\]
which is the integral-differential equation in (\ref{eq:f(s,x)}). Every
converging subsequence $\Lambda$ gives the same integral-differential
equation. The equation (\ref{eq:f(s,x)}) has a unique solution which satisfies
$f(0,x)=0$ and $\lim_{x\rightarrow\infty}xf(s,x)=s$, since it can be reduced
to a first order differential equation of Abel ($\sigma\neq0$) or Riccati
($\sigma=0$) type, see Propositions \ref{prop:4} and \ref{prop:5}. Hence
$f_{\Lambda}(s,x)$ is independent of the subsequence $\Lambda$.
\end{proof}

\begin{proposition}
\label{prop:4} We have%
\begin{equation}
f(t,x)=t^{1-\sigma}f(1,xt^{-\sigma}). \label{f1}%
\end{equation}

\end{proposition}

\begin{proof}
From (\ref{limit}), we have%
\[
f(t,x)=\underset{n\rightarrow\infty}{\lim}\frac{\phi\left(  n\right)  }%
{n}\frac{P_{\left[  nt\right]  }^{\prime}\left(  \phi\left(  n\right)
x\right)  }{P_{\left[  nt\right]  }\left(  \phi\left(  n\right)  x\right)  },
\]
which we can rewrite as%
\begin{equation}
f(t,x)=\underset{n\rightarrow\infty}{\lim}\frac{\left[  nt\right]  }{n}%
\frac{\phi\left(  n\right)  }{\phi\left(  \left[  nt\right]  \right)  }%
\frac{\phi\left(  \left[  nt\right]  \right)  }{\left[  nt\right]  }%
\frac{P_{\left[  nt\right]  }^{\prime}\left(  \phi\left(  n\right)  x\right)
}{P_{\left[  nt\right]  }\left(  \phi\left(  n\right)  x\right)  }.
\label{eq1}%
\end{equation}
Using (\ref{RV}), we have%
\[
\phi\left(  \left[  nt\right]  \right)  t^{-\sigma}\sim\phi\left(  n\right)
,\quad n\rightarrow\infty.
\]
Hence,%
\[
\underset{n\rightarrow\infty}{\lim}\frac{\phi\left(  \left[  nt\right]
\right)  }{\left[  nt\right]  }\frac{P_{n}^{\prime}\left(  \phi\left(  \left[
nt\right]  \right)  xt^{-\sigma}\right)  }{P_{n}\left(  \phi\left(  \left[
nt\right]  \right)  xt^{-\sigma}\right)  }=f(1,xt^{-\sigma}),
\]
and we get%
\[
\underset{n\rightarrow\infty}{\lim}\frac{\left[  nt\right]  }{n}\frac
{\phi\left(  n\right)  }{\phi\left(  \left[  nt\right]  \right)  }\frac
{\phi\left(  \left[  nt\right]  \right)  }{\left[  nt\right]  }\frac
{P_{n}^{\prime}\left(  \phi\left(  \left[  nt\right]  \right)  xt^{-\sigma
}\right)  }{P_{n}\left(  \phi\left(  \left[  nt\right]  \right)  xt^{-\sigma
}\right)  }=t^{1-\sigma}f(1,xt^{-\sigma}).
\]
But from (\ref{eq1}), we conclude that%
\[
f(t,x)=t^{1-\sigma}f(1,xt^{-\sigma}).
\]

\end{proof}

\section{Abel and Riccati differential equations$\allowbreak$}

\begin{proposition}
\label{prop:5} The function
\begin{equation}
S(x)=f(1,x) \label{S}%
\end{equation}
satisfies the nonlinear ODE%
\begin{equation}
\left(  1-\sigma\right)  S(z)-\sigma zS^{\prime}(z)=\frac{a^{\prime
}(z)S(z)+a(z)S^{\prime}(z)+b^{\prime}\left(  z\right)  }{a\left(  z\right)
S(z)+b\left(  z\right)  } \label{ODE}%
\end{equation}
with boundary condition%
\begin{equation}
\underset{z\rightarrow\infty}{\lim}zS(z)=1. \label{BC}%
\end{equation}

\end{proposition}

\begin{proof}
Using (\ref{f1}) in (\ref{eq:f(s,x)}), we have%
\[
s^{1-\sigma}f(1,xs^{-\sigma})=\int\limits_{0}^{s}\frac{d}{dx}\ln\left[
t^{2\sigma-1}a\left(  xt^{-\sigma}\right)  t^{1-\sigma}S(xt^{-\sigma
})+t^{\sigma}b\left(  xt^{-\sigma}\right)  \right]  dt.
\]
Thus,%
\[
s^{1-\sigma}S(xs^{-\sigma})=\int\limits_{0}^{s}\frac{d}{dx}\ln\left[  a\left(
xt^{-\sigma}\right)  S(xt^{-\sigma})+b\left(  xt^{-\sigma}\right)  \right]
dt.
\]
Differentiation with respect to $s$ gives%
\[
\left(  1-\sigma\right)  s^{-\sigma}S(xs^{-\sigma})-\sigma xs^{-2\sigma
}S^{\prime}(xs^{-\sigma})=\frac{d}{dx}\ln\left[  a\left(  xs^{-\sigma}\right)
S(xs^{-\sigma})+b\left(  xs^{-\sigma}\right)  \right]
\]
or, equivalently,%
\begin{align*}
&  \left(  1-\sigma\right)  S(xs^{-\sigma})-\sigma xs^{-\sigma}S^{\prime
}(xs^{-\sigma})\\
&  =\frac{a^{\prime}(xs^{-\sigma})S(xs^{-\sigma})+a(xs^{-\sigma})S^{\prime
}(xs^{-\sigma})+b^{\prime}\left(  xs^{-\sigma}\right)  }{a\left(  xs^{-\sigma
}\right)  S(xs^{-\sigma})+b\left(  xs^{-\sigma}\right)  }.
\end{align*}
Introducing the new variable%
\[
z=xs^{-\sigma},
\]
we get%
\[
(1-\sigma)S(z)-\sigma zS^{\prime}(z)=\frac{a^{\prime}(z)S(z)+a(z)S^{\prime
}(z)+b^{\prime}\left(  z\right)  }{a\left(  z\right)  S(z)+b\left(  z\right)
}.
\]
Finally, the boundary condition (\ref{largex}) implies
\[
\underset{z\rightarrow\infty}{\lim}zS(z)=1.
\]

\end{proof}

Solving for $S^{\prime}(z)$ in (\ref{ODE}), we get%
\begin{equation}
\left[  z\sigma\left(  aS+b\right)  +a\right]  S^{\prime}=(1-\sigma)S\left(
aS+b\right)  -\left(  b^{\prime}+Sa^{\prime}\right)  . \label{abel0}%
\end{equation}

If $\sigma=0,$ then%
\begin{equation}
S^{\prime}=S^{2}+\frac{b(z)-a^{\prime}(z)}{a(z)}S-\frac{b^{\prime}(z)}{a(z)}.
\label{riccatti}%
\end{equation}
Thus, in this case $S(z)$ is the solution of a Riccati equation
\cite{MR0357936}. The substitution
\begin{equation}
u(z)=\exp\left[  -%
%TCIMACRO{\dint \limits^{z}}%
%BeginExpansion
{\displaystyle\int\limits^{z}}
%EndExpansion
S(t)dt\right]  \label{u}%
\end{equation}
reduces (\ref{riccatti}) to a second-order linear ODE%
\[
a(z)u^{\prime\prime}+\left[  a^{\prime}(z)-b(z)\right]  u^{\prime}-b^{\prime
}(z)u=0,
\]
which can be rewritten as%
\[
\left[  a(z)u^{\prime}-b(z)u\right]  ^{\prime}=0.
\]
Thus,%
\[
a(z)u^{\prime}-b(z)u=C_{1},
\]
and therefore
\begin{equation}
u(z)=\frac{1}{h(z)}\left[  C_{1}%
%TCIMACRO{\dint \limits^{z}}%
%BeginExpansion
{\displaystyle\int\limits^{z}}
%EndExpansion
\frac{h(t)}{a(t)}dt+C_{2}\right]  , \label{u1}%
\end{equation}
where%
\[
h(z)=\exp\left[  -%
%TCIMACRO{\dint \limits^{z}}%
%BeginExpansion
{\displaystyle\int\limits^{z}}
%EndExpansion
\frac{b(t)}{a(t)}dt\right]  .
\]
Using (\ref{u1}) in (\ref{u}), we get%
\begin{equation}
S(z)=-\frac{d}{dz}\ln\left[  u(z)\right]  =-\frac{b(z)}{a(z)}-\frac
{h(z)}{a(z)}\left[
%TCIMACRO{\dint \limits^{z}}%
%BeginExpansion
{\displaystyle\int\limits^{z}}
%EndExpansion
\frac{h(t)}{a(t)}dt+C\right]  ^{-1}. \label{solRicat}%
\end{equation}

Alternatively, we note that
\begin{equation}
S_{p}(z)=-\frac{b(z)}{a(z)} \label{particular}%
\end{equation}
is a particular solution of the Riccati equation (\ref{riccatti}). Thus, we
can set \cite{MR2001201}%
\begin{equation}
S(z)=S_{p}(z)+\frac{1}{y(z)} \label{subspart}%
\end{equation}
in (\ref{riccatti}) and obtain the linear equation%
\begin{equation}
y^{\prime}-\frac{a^{\prime}+b}{a}y=-1, \label{riccati1}%
\end{equation}
with solution
\begin{equation}
y(z)=-u(z)\left[
%TCIMACRO{\dint \limits^{z}}%
%BeginExpansion
{\displaystyle\int\limits^{z}}
%EndExpansion
\frac{1}{u(t)}dt+C\right]  , \label{solpart}%
\end{equation}
where%
\[
u(z)=\exp\left[
%TCIMACRO{\dint \limits^{z}}%
%BeginExpansion
{\displaystyle\int\limits^{z}}
%EndExpansion
\frac{a^{\prime}\left(  t\right)  +b\left(  t\right)  }{a\left(  t\right)
}dt\right]  =a\left(  z\right)  \exp\left[
%TCIMACRO{\dint \limits^{z}}%
%BeginExpansion
{\displaystyle\int\limits^{z}}
%EndExpansion
\frac{b\left(  t\right)  }{a\left(  t\right)  }dt\right]  .
\]
Using (\ref{solpart}) in (\ref{subspart}), we recover (\ref{solRicat}).

If $\sigma\neq0,$ we have
\begin{equation}
\left[  S+g\left(  z\right)  \right]  S^{\prime}=f_{2}\left(  z\right)
S^{2}+f_{1}\left(  z\right)  S+f_{0}\left(  z\right)  \label{abel}%
\end{equation}
where%
\begin{align}
g\left(  z\right)   &  =\frac{a(z)+\sigma zb(z)}{\sigma za(z)},\quad
f_{0}\left(  z\right)  =-\frac{b^{\prime}(z)}{\sigma za(z)},\label{gf}\\
f_{1}\left(  z\right)   &  =\frac{(1-\sigma)b(z)-a^{\prime}(z)}{\sigma
za(z)},\quad f_{2}\left(  z\right)  =\frac{1-\sigma}{\sigma z}.\nonumber
\end{align}
Differential equations of the form (\ref{abel}) are called Abel equations of
the second kind \cite{MR0021170}. The substitution%
\begin{equation}
w(z)=\left[  S(z)+g(z)\right]  E(z), \label{w-y}%
\end{equation}
where%
\begin{equation}
E(z)=\exp\left[  -%
%TCIMACRO{\dint \limits^{z}}%
%BeginExpansion
{\displaystyle\int\limits^{z}}
%EndExpansion
f_{2}\left(  t\right)  dt\right]  =z^{1-\sigma^{-1}}, \label{E}%
\end{equation}
transforms equation (\ref{abel}) to the canonical form%
\begin{equation}
w\left(  \frac{dw}{dx}-1\right)  =R(x), \label{canonical}%
\end{equation}
where%
\begin{equation}
R(z)=\frac{F_{0}\left(  z\right)  }{F_{1}\left(  z\right)  }, \label{R}%
\end{equation}
with%
\begin{align*}
F_{0}\left(  z\right)   &  =\left[  f_{0}\left(  z\right)  -f_{1}\left(
z\right)  g\left(  z\right)  +f_{2}\left(  z\right)  g^{2}\left(  z\right)
\right]  E^{2}(z),\\
F_{1}\left(  z\right)   &  =\left[  f_{1}\left(  z\right)  -2f_{2}\left(
z\right)  g\left(  z\right)  +g^{\prime}\left(  z\right)  \right]  E(z),
\end{align*}
and the new variable $x$ is defined by%
\begin{equation}
x=%
%TCIMACRO{\dint \limits^{z}}%
%BeginExpansion
{\displaystyle\int\limits^{z}}
%EndExpansion
F_{1}\left(  t\right)  dt. \label{x-z}%
\end{equation}
General solutions of (\ref{canonical}) for different functions $R(x)$ are
given in \cite{MR2001201}.

Once again, we note that (\ref{particular}) is a particular solution of the
Abel equation (\ref{abel0}). With this in mind, we can rewrite equation
(\ref{w-y}) in the form%
\[
S(z)=S_{p}(z)-\frac{1}{\sigma z}+z^{\frac{1}{\sigma}-1}w\left(  z\right)  .
\]
We can also use the particular solution (\ref{particular}) to construct a
self-transformation of the Abel equation (\ref{abel}). Setting
\[
y(z)=\frac{1}{S(z)-S_{p}(z)},
\]
we get \cite{MR2001201}%
\[
\left(  y+\frac{1}{S_{p}+g}\right)  y^{\prime}=\frac{S_{p}^{\prime}%
-f_{1}-2f_{2}S_{p}}{S_{p}+g}y^{2}-\frac{f_{2}}{S_{p}+g}y.
\]
From (\ref{gf}) we have%

\[
S_{p}+g=\frac{a(z)+\sigma zb(z)}{\sigma za(z)}-\frac{b(z)}{a(z)}=\frac
{1}{\sigma z},
\]
and
\[
\frac{S_{p}^{\prime}-f_{1}-2f_{2}S_{p}}{S_{p}+g}=\sigma zS_{p}^{\prime
}+\left(  \sigma-1\right)  S_{p}+\frac{a^{\prime}}{a}.
\]
Thus, we obtain the equation
\begin{equation}
\left(  y+\sigma z\right)  y^{\prime}=\left[  \sigma zS_{p}^{\prime}+\left(
\sigma-1\right)  S_{p}+\frac{a^{\prime}}{a}\right]  y^{2}+\left(
\sigma-1\right)  y. \label{abel1}%
\end{equation}
When $\sigma=0,$ equation (\ref{abel1}) reduces to the Riccati equation
(\ref{riccati1}).

Since there is no method that will allow us to solve the Abel equation
(\ref{abel0}) in general, we will construct a particular solution satisfying
the asymptotic condition (\ref{BC}).

\begin{proposition}
Suppose that%
\begin{subequations}
\begin{equation}
a(z)=a_{2}z^{2}+a_{1}z+a_{0},\quad b(z)=b_{1}z+b_{0}, \label{ab}%
\end{equation}
and that%
\end{subequations}
\[
-\frac{a_{2}+b_{1}}{\sigma\left(  a_{2}+b_{1}\right)  +a_{2}}\notin%
\mathbb{N},
\]
where $\mathbb{N}$ denotes the set of natural numbers. Then, the Abel equation
(\ref{abel0}) has the unique solution%
\begin{equation}
S(z)=%
%TCIMACRO{\dsum \limits_{n=0}^{\infty}}%
%BeginExpansion
{\displaystyle\sum\limits_{n=0}^{\infty}}
%EndExpansion
\frac{c_{n}}{z^{n}}, \label{sol}%
\end{equation}
where the coefficients $c_{n}$ are defined by the recurrence relation%
\begin{gather}
\left[  a_{2}\left(  n-1\right)  +\left(  \sigma n-\sigma+1\right)  \left(
a_{2}+b_{1}\right)  \right]  c_{n}=\left[  \left(  2\sigma-1-\sigma n\right)
\left(  a_{1}+b_{0}\right)  -a_{1}\left(  n-2\right)  \right]  c_{n-1}%
\label{req}\\
-a_{0}\left(  n-2\right)  c_{n-2}-%
%TCIMACRO{\dsum \limits_{k=0}^{n-3}}%
%BeginExpansion
{\displaystyle\sum\limits_{k=0}^{n-3}}
%EndExpansion
\left(  \sigma k+1\right)  c_{k+1}\left(  a_{0}c_{n-2-k}+a_{1}c_{n-1-k}%
\right)  +a_{2}\left(  \sigma k+\sigma+1\right)  c_{k+2}c_{n-1-k},\nonumber
\end{gather}
for $n=2,3,\ldots,$ with $c_{0}=0$ and $c_{1}=1.$
\end{proposition}

\begin{proof}
Setting $z=x^{-1}$ in (\ref{abel0}), we have%
\begin{equation}
-x^{2}\left[  \sigma x^{-1}\left(  aS+b\right)  +a\right]  S^{\prime
}=(1-\sigma)S\left(  aS+b\right)  -\left(  b^{\prime}+Sa^{\prime}\right)  ,
\label{abel2}%
\end{equation}
where $a,b,a^{\prime}$ and $b^{\prime}$ are evaluated at $x^{-1}.$ Using
(\ref{ab}) and (\ref{sol}) in (\ref{abel2}), we obtain
\begin{gather}
-b_{1}x^{2}+a_{0}%
%TCIMACRO{\dsum \limits_{n=3}^{\infty}}%
%BeginExpansion
{\displaystyle\sum\limits_{n=3}^{\infty}}
%EndExpansion
\left(  n-3\right)  c_{n-3}x^{n}\nonumber\\%
%TCIMACRO{\dsum \limits_{n=2}^{\infty}}%
%BeginExpansion
{\displaystyle\sum\limits_{n=2}^{\infty}}
%EndExpansion
\left[  \left(  n-2\right)  \left(  \sigma b_{0}+a_{1}\right)  -\left(  \sigma
b_{0}-b_{0}+a_{1}\right)  \right]  c_{n-2}x^{n}\nonumber\\
+%
%TCIMACRO{\dsum \limits_{n=1}^{\infty}}%
%BeginExpansion
{\displaystyle\sum\limits_{n=1}^{\infty}}
%EndExpansion
\left[  \left(  n-1\right)  \left(  \sigma b_{1}+a_{2}\right)  -\left(  \sigma
b_{1}-b_{1}+2a_{2}\right)  \right]  c_{n-1}x^{n}\nonumber\\
+a_{0}%
%TCIMACRO{\dsum \limits_{n=2}^{\infty}}%
%BeginExpansion
{\displaystyle\sum\limits_{n=2}^{\infty}}
%EndExpansion%
%TCIMACRO{\dsum \limits_{k=2}^{n}}%
%BeginExpansion
{\displaystyle\sum\limits_{k=2}^{n}}
%EndExpansion
\left[  \sigma\left(  k-2\right)  -\left(  \sigma-1\right)  \right]
c_{k-2}c_{n-k}\ x^{n}\label{req1}\\
+a_{1}%
%TCIMACRO{\dsum \limits_{n=1}^{\infty}}%
%BeginExpansion
{\displaystyle\sum\limits_{n=1}^{\infty}}
%EndExpansion%
%TCIMACRO{\dsum \limits_{k=1}^{n}}%
%BeginExpansion
{\displaystyle\sum\limits_{k=1}^{n}}
%EndExpansion
\left[  \sigma\left(  k-1\right)  -\left(  \sigma-1\right)  \right]
c_{k-1}c_{n-k}\ x^{n}\nonumber\\
+a_{2}%
%TCIMACRO{\dsum \limits_{n=0}^{\infty}}%
%BeginExpansion
{\displaystyle\sum\limits_{n=0}^{\infty}}
%EndExpansion%
%TCIMACRO{\dsum \limits_{k=0}^{n}}%
%BeginExpansion
{\displaystyle\sum\limits_{k=0}^{n}}
%EndExpansion
\left[  \sigma k-\left(  \sigma-1\right)  \right]  c_{k}c_{n-k}\ x^{n}%
=0.\nonumber
\end{gather}
Comparing coefficients of $x$ we get, up to order $x,$%
\begin{align*}
a_{2}\left(  1-\sigma\right)  c_{0}^{2}  &  =0,\\
-\left(  \sigma b_{1}-b_{1}+2a_{2}\right)  c_{0}-a_{1}\left(  \sigma-1\right)
c_{0}^{2}  &  =0,
\end{align*}
and therefore $c_{0}=0.$ The next equation is%
\[
-b_{1}+\left(  b_{1}-a_{2}\right)  c_{1}+a_{2}c_{1}^{2}=0
\]
or
\[
\left(  c_{1}-1\right)  \left(  b_{1}+a_{2}c_{1}\right)  =0,
\]
and hence, $c_{1}=1.$ Using $c_{0}=0$ in (\ref{req1}), we get after some
simplification%
\begin{gather*}
\left[  \left(  n-2\right)  \left(  \sigma b_{0}+a_{1}\right)  -\left(  \sigma
b_{0}-b_{0}+a_{1}\right)  \right]  c_{n-2}\\
+\left[  \left(  n-1\right)  \left(  \sigma b_{1}+a_{2}\right)  -\left(
\sigma b_{1}-b_{1}+2a_{2}\right)  \right]  c_{n-1}\\
+a_{0}%
%TCIMACRO{\dsum \limits_{k=3}^{n-1}}%
%BeginExpansion
{\displaystyle\sum\limits_{k=3}^{n-1}}
%EndExpansion
\left[  \sigma\left(  k-3\right)  +1\right]  c_{k-2}c_{n-k}\ +a_{1}%
%TCIMACRO{\dsum \limits_{k=2}^{n-1}}%
%BeginExpansion
{\displaystyle\sum\limits_{k=2}^{n-1}}
%EndExpansion
\left[  \sigma\left(  k-2\right)  +1\right]  c_{k-1}c_{n-k}\ \\
+a_{0}\left(  n-3\right)  c_{n-3}+a_{2}%
%TCIMACRO{\dsum \limits_{k=1}^{n-1}}%
%BeginExpansion
{\displaystyle\sum\limits_{k=1}^{n-1}}
%EndExpansion
\left[  \sigma\left(  k-1\right)  +1\right]  c_{k}c_{n-k}\ =0.
\end{gather*}
Shifting $n$ to $n+1,$ we have
\begin{gather}
\left[  \left(  n-1\right)  \left(  \sigma b_{0}+a_{1}\right)  -\left(  \sigma
b_{0}-b_{0}+a_{1}\right)  \right]  c_{n-1}\nonumber\\
+\left[  n\left(  \sigma b_{1}+a_{2}\right)  -\left(  \sigma b_{1}%
-b_{1}+2a_{2}\right)  \right]  c_{n}\label{req2}\\
+a_{0}%
%TCIMACRO{\dsum \limits_{k=3}^{n}}%
%BeginExpansion
{\displaystyle\sum\limits_{k=3}^{n}}
%EndExpansion
\left[  \sigma\left(  k-3\right)  +1\right]  c_{k-2}c_{n+1-k}\ +a_{1}%
%TCIMACRO{\dsum \limits_{k=2}^{n}}%
%BeginExpansion
{\displaystyle\sum\limits_{k=2}^{n}}
%EndExpansion
\left[  \sigma\left(  k-2\right)  +1\right]  c_{k-1}c_{n+1-k}\ \nonumber\\
+a_{0}\left(  n-2\right)  c_{n-2}+a_{2}%
%TCIMACRO{\dsum \limits_{k=1}^{n}}%
%BeginExpansion
{\displaystyle\sum\limits_{k=1}^{n}}
%EndExpansion
\left[  \sigma\left(  k-1\right)  +1\right]  c_{k}c_{n+1-k}\ =0.\nonumber
\end{gather}
Using $c_{1}=1$ in (\ref{req2}), we obtain
\begin{gather}
\left[  \left(  n-1\right)  \left(  \sigma b_{0}+a_{1}\right)  -\left(  \sigma
b_{0}-b_{0}+a_{1}\right)  +a_{1}\left(  n\sigma-2\sigma+1\right)  \right]
c_{n-1}\nonumber\\
+\left[  n\left(  \sigma b_{1}+a_{2}\right)  -\left(  \sigma b_{1}%
-b_{1}+2a_{2}\right)  +a_{2}\left(  n\sigma-\sigma+2\allowbreak\right)
\right]  c_{n}\label{req3}\\
+a_{0}%
%TCIMACRO{\dsum \limits_{k=3}^{n}}%
%BeginExpansion
{\displaystyle\sum\limits_{k=3}^{n}}
%EndExpansion
\left[  \sigma\left(  k-3\right)  +1\right]  c_{k-2}c_{n+1-k}\ +a_{1}%
%TCIMACRO{\dsum \limits_{k=2}^{n-1}}%
%BeginExpansion
{\displaystyle\sum\limits_{k=2}^{n-1}}
%EndExpansion
\left[  \sigma\left(  k-2\right)  +1\right]  c_{k-1}c_{n+1-k}\ \nonumber\\
+a_{0}\left(  n-2\right)  c_{n-2}+a_{2}%
%TCIMACRO{\dsum \limits_{k=2}^{n-1}}%
%BeginExpansion
{\displaystyle\sum\limits_{k=2}^{n-1}}
%EndExpansion
\left[  \sigma\left(  k-1\right)  +1\right]  c_{k}c_{n+1-k}\ =0.\nonumber
\end{gather}
Shifting the sums in (\ref{req3}) we conclude that%
\begin{gather*}
\left[  \left(  n-1\right)  \left(  \sigma b_{0}+a_{1}\right)  -\left(  \sigma
b_{0}-b_{0}+a_{1}\right)  +a_{1}\left(  n\sigma-2\sigma+1\right)  \right]
c_{n-1}\\
+\left[  n\left(  \sigma b_{1}+a_{2}\right)  -\left(  \sigma b_{1}%
-b_{1}+2a_{2}\right)  +a_{2}\left(  n\sigma-\sigma+2\allowbreak\right)
\right]  c_{n}\\
+a_{0}%
%TCIMACRO{\dsum \limits_{k=0}^{n-3}}%
%BeginExpansion
{\displaystyle\sum\limits_{k=0}^{n-3}}
%EndExpansion
\left(  \sigma k+1\right)  c_{k+1}c_{n-2-k}\ +a_{1}%
%TCIMACRO{\dsum \limits_{k=0}^{n-3}}%
%BeginExpansion
{\displaystyle\sum\limits_{k=0}^{n-3}}
%EndExpansion
\left(  \sigma k+1\right)  c_{k+1}c_{n-1-k}\ \\
+a_{0}\left(  n-2\right)  c_{n-2}+a_{2}%
%TCIMACRO{\dsum \limits_{k=0}^{n-3}}%
%BeginExpansion
{\displaystyle\sum\limits_{k=0}^{n-3}}
%EndExpansion
\left[  \sigma\left(  k+1\right)  +1\right]  c_{k+2}c_{n-1-k}\ =0,
\end{gather*}
and the result follows.
\end{proof}

\begin{remark}
Note that in the previous proof $c_{0}$ and $c_{1}$ were \emph{obtained }as
part of the process of finding a unique solution of the differential equation
(\ref{abel0}). We didn't need to assume their values at all.
\end{remark}

\section{The Stieltjes transform}

We will now show that the function $S(z)$ that we obtained in the previous
section is the Stieltjes transform of the equilibrium measure for the
polynomials $P_{n}(x).$

From (\ref{limit}), (\ref{Partial}) and (\ref{S}), we know that%
\[
S(z)=f(1,z)=\lim_{n\rightarrow\infty}\frac{\phi(n)}{n}\frac{P_{n}^{\prime
}\left[  \phi(n)z\right]  }{P_{n}\left[  \phi(n)z\right]  }=\lim
_{n\rightarrow\infty}\frac{1}{n}\sum_{k=1}^{n}\frac{1}{z-z_{k,n}},
\]
where%
\[
z_{k,n}=\frac{x_{k,n}}{\phi(n)},
\]
and $\left\{  x_{k,n}\right\}  _{k=1}^{n}$ are the zeros of $P_{n}(x).$
Introducing the zero counting measures \cite[Section 1.2]{MR903848}%
\begin{equation}
\psi_{n}\left(  t\right)  =\left\{
\begin{array}
[c]{c}%
0,\quad t\leq z_{1,n}\\
\frac{k}{n},\quad z_{k,n}<t\leq z_{k+1,n},\quad k=1,2,\ldots,n-1\\
1,\quad t\geq z_{n,n}%
\end{array}
\right.  , \label{psin}%
\end{equation}
we have
\[
\frac{1}{n}\sum_{k=1}^{n}\frac{1}{z-z_{k,n}}=%
%TCIMACRO{\dint \limits_{-\infty}^{\infty}}%
%BeginExpansion
{\displaystyle\int\limits_{-\infty}^{\infty}}
%EndExpansion
\frac{1}{z-t}d\psi_{n}\left(  t\right)  ,
\]
and there exist a probability measure $\psi\left(  t\right)  $ such that%
\begin{equation}
S(z)=%
%TCIMACRO{\dint \limits_{-\infty}^{\infty}}%
%BeginExpansion
{\displaystyle\int\limits_{-\infty}^{\infty}}
%EndExpansion
\frac{1}{z-t}d\psi\left(  t\right)  . \label{Stieltjes}%
\end{equation}
The integral above is called the \textit{Stieltjes transform} \cite[Section
65]{MR0025596} of $\psi\left(  t\right)  $, and $\psi\left(  t\right)  $ is
called the \textit{equilibrium measure}. To recover the measure $\psi\left(
t\right)  $ from (\ref{Stieltjes}), we can use the \textit{Stieltjes-Perron
inversion formula}
\begin{equation}
\left[  \psi\right]  \left(  s\right)  -\left[  \psi\right]  \left(  t\right)
=\frac{1}{\pi}\lim_{y\rightarrow0^{+}}%
%TCIMACRO{\dint \limits_{s}^{t}}%
%BeginExpansion
{\displaystyle\int\limits_{s}^{t}}
%EndExpansion
\operatorname{Im}\left[  S\left(  x+\mathrm{i}y\right)  dx\right]  ,
\label{Inversion}%
\end{equation}
where $\left[  \psi\right]  $ denotes the jump operator%
\[
\left[  \psi\right]  \left(  s\right)  =\frac{1}{2}\left[  \lim_{t\rightarrow
s^{-}}\psi\left(  t\right)  +\lim_{t\rightarrow s^{+}}\psi\left(  t\right)
\right]  .
\]
Note that
\[
\psi\left(  t\right)  =Au\left(  s-a\right)  +\nu\left(  t\right)  ,
\]
($u$ is Heaviside's step function) if and only if%
\[
S(z)=\frac{A}{z-a}+F(z).
\]
In particular, the absolutely continuous part of $\psi$ is given by%
\begin{equation}
\psi^{\prime}\left(  t\right)  =-\frac{1}{\pi}\lim_{y\rightarrow0^{+}%
}\operatorname{Im}S\left(  t+\mathrm{i}y\right)  . \label{ACInversion}%
\end{equation}

The function $S(z)$ has the asymptotic behavior \cite[Section 12.9]%
{MR0453984}
\[
S(z)\sim\frac{\mu_{0}}{z}+\frac{\mu_{1}}{z^{2}}+\frac{\mu_{2}}{z^{3}}%
+\cdots,\quad z\rightarrow\infty,
\]
where the coefficients $\mu_{n}$ are the \textit{moments} of the measure
$\psi\left(  t\right)  $%
\[
\mu_{n}=%
%TCIMACRO{\dint \limits_{-\infty}^{\infty}}%
%BeginExpansion
{\displaystyle\int\limits_{-\infty}^{\infty}}
%EndExpansion
t^{n}d\psi\left(  t\right)  ,\quad n=0,1,\ldots.
\]

\section{Examples}

\subsection{Jacobi polynomials}

The Jacobi polynomials $P_{n}^{(\alpha,\beta)}$ are orthogonal polynomials on
$[-1,1]$ satisfying
\[%
%TCIMACRO{\dint \limits_{-1}^{1}}%
%BeginExpansion
{\displaystyle\int\limits_{-1}^{1}}
%EndExpansion
P_{n}^{(\alpha,\beta)}(x)P_{m}^{(\alpha,\beta)}(x)(1-x)^{\alpha}(1+x)^{\beta
}\,dx=\frac{2^{\alpha+\beta+1}}{2n+\alpha+\beta+1}\frac{\Gamma(n+\alpha
+1)\Gamma(n+\beta+1)}{\Gamma(n+\alpha+\beta+1)n!}\ \delta_{m,n}.
\]
here we take $\alpha,\beta>-1$ in order that the weight is integrable on
$[-1,1]$. They satisfy the relation \cite[Eq/ (4.5.7) on p.~72]{MR0372517}
\begin{align}
(2n+\alpha+\beta+2)(1-x^{2})\frac{d}{dx}P_{n}^{(\alpha,\beta)}(x)  &
=(n+\alpha+\beta+1)[(2n+\alpha+\beta+2)x+\alpha-\beta]P_{n}^{(\alpha,\beta
)}(x)\label{eq:Jac}\\
&  -\ 2(n+1)(n+\alpha+\beta+1)P_{n+1}^{(\alpha,\beta)}(x).\nonumber
\end{align}
The monic polynomials are
\[
P_{n}(x)=\frac{2^{n}n!\Gamma(2n+\alpha+\beta)}{\Gamma(2n+\alpha/\beta
+1)}\ P_{n}^{(\alpha,\beta)}(x)
\]
and the relation (\ref{eq:Jac}) become
\begin{equation}
(2n+\alpha+\beta+1)P_{n+1}(x)=(x^{2}-1)P_{n}^{\prime}(x)+(n+\alpha
+\beta+1)\left(  x+\frac{\alpha-\beta}{2n+\alpha+\beta+2}\right)  P_{n}(x).
\label{eq:Jacmon}%
\end{equation}
This of the form (\ref{eq:diff2}) with
\[
A_{n}(x)=\frac{x^{2}-1}{2n+\alpha+\beta+1},\quad B_{n}(x)=\frac{n+\alpha
+\beta+1}{2n+\alpha+\beta+1}\left(  x+\frac{\alpha-\beta}{2n+\alpha+\beta
+2}\right)  .
\]
All the zeros of Jacobi polynomials are on $[-1,1]$ and they are interlacing.
Hence we need no scaling and can take $\phi(n)=1$ for all $n$. Clearly
$\sigma=0$ and
\begin{align}
\lim_{n\rightarrow\infty}nA_{n}(x)  &  =\frac{1}{2}(x^{2}%
-1)=a(x),\label{abJacobi}\\
\lim_{n\rightarrow\infty}B_{n}(x)  &  =\frac{x}{2}=b(x).\nonumber
\end{align}
It follows that%
\begin{equation}
h(x)=\exp\left[  -%
%TCIMACRO{\dint \limits^{x}}%
%BeginExpansion
{\displaystyle\int\limits^{x}}
%EndExpansion
\frac{t}{t^{2}-1}dt\right]  =\frac{1}{\sqrt{x^{2}-1}}. \label{hJacobi}%
\end{equation}
Using (\ref{abJacobi}) and (\ref{hJacobi}) in (\ref{solRicat}), we get%
\[
S(x)=-\frac{x}{x^{2}-1}+\frac{1}{x^{2}-1}\frac{1}{x+C\sqrt{x^{2}-1}}.
\]
In order that $\lim_{x\rightarrow\infty}xS(x)=1$ we need to choose $C=-1$,
which gives
\[
f(1,x)=S(x)=\frac{1}{\sqrt{x^{2}-1}}.
\]
This function is analytic in $\mathbb{C}\setminus\lbrack-1,1]$ and is the
Stieltjes transform of a positive measure:
\[
\frac{1}{\sqrt{x^{2}-1}}=\frac{1}{\pi}%
%TCIMACRO{\dint \limits_{-1}^{1}}%
%BeginExpansion
{\displaystyle\int\limits_{-1}^{1}}
%EndExpansion
\frac{1}{x-t}\frac{dt}{\sqrt{1-t^{2}}},\quad x\notin\lbrack-1,1],
\]
so that the asymptotic distribution of the zeros of Jacobi polynomials is
given by
\[
\lim_{n\rightarrow\infty}\frac{1}{n}\#\{\text{zeros of }P_{n}^{(\alpha,\beta
)}\text{ in }[a,b]\}=\frac{1}{\pi}%
%TCIMACRO{\dint \limits_{a}^{b}}%
%BeginExpansion
{\displaystyle\int\limits_{a}^{b}}
%EndExpansion
\frac{dt}{\sqrt{1-t^{2}}},\qquad\lbrack a,b]\subset\lbrack-1,1].
\]

\subsection{Laguerre polynomials}

Laguerre polynomials $L_{n}^{\alpha}$ are orthogonal polynomials on
$[0,\infty)$
\[%
%TCIMACRO{\dint \limits_{0}^{\infty}}%
%BeginExpansion
{\displaystyle\int\limits_{0}^{\infty}}
%EndExpansion
L_{n}^{\alpha}(x)L_{m}^{\alpha}(x)x^{\alpha}e^{-x}\,dx=\frac{\Gamma
(n+\alpha+1)}{n!}\ \delta_{m,n},
\]
where $\alpha>-1$. The zeros are real, positive and interlacing. From
Szeg\H{o} \cite[Eq. (5.1.14)]{MR0372517} we learn that
\[
x\frac{d}{dx}L_{n}^{\alpha}(x)=nL_{n}^{\alpha}(x)-(n+\alpha)L_{n-1}^{\alpha
}(x).
\]
Combined with the recurrence relation \cite[Eq. (5.1.10)]{MR0372517}
\[
(n+1)L_{n+1}^{\alpha}(x)=(-x+2n+\alpha+1)L_{n}^{\alpha}(x)-(n+\alpha
)L_{n-1}^{\alpha}(x)
\]
this gives the relation
\[
x\frac{d}{dx}L_{n}^{\alpha}(x)=(x-n-\alpha-1)L_{n}^{\alpha}(x)+(n+1)L_{n+1}%
^{\alpha}(x).
\]
For the monic polynomials $P_{n}=(-1)^{n}n!L_{n}^{\alpha}$ we then find
\begin{equation}
P_{n+1}(x)=-xP_{n}^{\prime}(x)+(x-n-\alpha-1)P_{n}(x), \label{eq:Lag}%
\end{equation}
which is of the form (\ref{eq:diff2}) with
\[
A_{n}(x)=-x,\qquad B_{n}(x)=x-n-\alpha-1.
\]
In order that (\ref{eq:Aa2})--(\ref{eq:Bb2}) holds, we choose the scaling
$\phi(n)=n$, so that
\begin{equation}
a(x)=-x,\qquad b(x)=x-1 \label{abLaguerre}%
\end{equation}
and $\sigma=1.$ Using these in (\ref{w-y}), (\ref{E}), (\ref{R}) and
(\ref{x-z}), we have%
\begin{equation}
w(z)=S(z)-1+\frac{2}{z},\quad E(z)=1,\quad R(z)=-\frac{2}{3z},\quad x=\frac
{3}{z}. \label{wLaguerre}%
\end{equation}
Hence, $R(x)=-\frac{2}{9}x$ and the canonical form of the Abel equation
(\ref{canonical}) is%
\[
w\left(  \frac{dw}{dx}-1\right)  =-\frac{2}{9}x,
\]
and therefore%
\[
w(x)=\frac{2Cx+1\pm\sqrt{2Cx+1}}{3C}.
\]
Using (\ref{wLaguerre}), we get%
\[
S(z)=\frac{3C+1\pm\sqrt{6Cz^{-1}+1}}{3C},
\]
which gives%
\[
0=\lim_{z\rightarrow\infty}S(z)=\frac{3C+1\pm1}{3C}.
\]
Thus, the desired solution has the positive sign and $C=-\frac{2}{3}$
\[
f(1,z)=S(z)=\frac{1-\sqrt{1-4z^{-1}}}{2}=\frac{2}{z+\sqrt{z^{2}-4z}}.
\]
This is the Stieltjes transform of a positive measure on $[0,4]$
\[
\frac{1}{2\pi}%
%TCIMACRO{\dint \limits_{0}^{4}}%
%BeginExpansion
{\displaystyle\int\limits_{0}^{4}}
%EndExpansion
\frac{\sqrt{4-t}}{\sqrt{t}}\frac{dt}{z-t}=\frac{2}{z+\sqrt{z^{2}-4z}},\quad
z\notin\lbrack0,4],
\]
so that we can conclude
\[
\lim_{n\rightarrow\infty}\frac{1}{n}\#\{\text{zeros of }L_{n}^{\alpha
}(nx)\text{ in }[a,b]\}=\frac{1}{2\pi}%
%TCIMACRO{\dint \limits_{a}^{b}}%
%BeginExpansion
{\displaystyle\int\limits_{a}^{b}}
%EndExpansion
\frac{\sqrt{4-t}}{\sqrt{t}}\,dt,\qquad\lbrack a,b]\subset\lbrack0,4].
\]
This result is not new and can be found using general methods based on
potential theory or general results for orthogonal polynomials defined by
their recurrence relation. The approach using Theorem \ref{thm:main} is new.

\subsection{Hermite polynomials \label{Hermite}}

Hermite polynomials are orthogonal on the real line with respect to the normal
distribution:
\[%
%TCIMACRO{\dint \limits_{-\infty}^{\infty}}%
%BeginExpansion
{\displaystyle\int\limits_{-\infty}^{\infty}}
%EndExpansion
H_{n}(x)H_{m}(x)e^{-x^{2}}\,dx=2^{n}\sqrt{\pi}n\delta_{m,n}.
\]
They satisfy the following differential-difference equation \cite[Eq.
(5.5.10)]{MR0372517}
\[
H_{n+1}(x)=-H_{n}^{\prime}(x)+2xH_{n}(x).
\]
The choice $\phi(n)=\sqrt{n}$ gives a scaling such that (\ref{eq:Aa2}%
)--(\ref{eq:Bb2}) result in
\begin{equation}
a(x)=-1,\quad b(x)=2x. \label{abHermite}%
\end{equation}
Using (\ref{abHermite}) and $\sigma=1/2$ in (\ref{w-y}), (\ref{E}), (\ref{R})
and (\ref{x-z}), we have%
\begin{equation}
w(z)=\frac{1}{z}S(z)-2+\frac{2}{z^{2}},\quad E(z)=\frac{1}{z},\quad
R(z)=-\frac{2}{3z^{2}},\quad x=\frac{3}{z^{2}}. \label{wHermite}%
\end{equation}
Hence, $R(x)=-\frac{2}{9}x$ and the canonical form of the Abel equation
(\ref{canonical}) is the same as the one for the Laguerre polynomials%
\[
w\left(  \frac{dw}{dx}-1\right)  =-\frac{2}{9}x,
\]
and therefore%
\[
S(z)=\frac{6Cz^{-2}+1\pm\sqrt{6Cz^{-2}+1}}{3C}z+2z-\frac{2}{z}.
\]
This behaves for $z\rightarrow\infty$ as
\[
\frac{1\pm1}{3C}z+2z,
\]
hence the required solution corresponds to the positive sign and $C=-\frac
{1}{3}$, giving
\[
S(z)=\left(  1-\sqrt{1-2z^{-2}}\right)  z=\frac{2}{z+\sqrt{z^{2}-2}}.
\]
This is the Stieltjes transform of a measure on $[-\sqrt{2},\sqrt{2}]$
\[
\frac{2}{z+\sqrt{z^{2}-2}}=\frac{1}{\pi}%
%TCIMACRO{\dint \limits_{-\sqrt{2}}^{\sqrt{2}}}%
%BeginExpansion
{\displaystyle\int\limits_{-\sqrt{2}}^{\sqrt{2}}}
%EndExpansion
\frac{\sqrt{2-t^{2}}}{z-t}\,dt,\quad z\notin\lbrack-\sqrt{2},\sqrt{2}],
\]
and consequently
\[
\lim_{n\rightarrow\infty}\frac{1}{n}\#\{\text{zeros of }H_{n}(\sqrt{n}x)\text{
in }[a,b]\}=\frac{1}{\pi}%
%TCIMACRO{\dint \limits_{a}^{b}}%
%BeginExpansion
{\displaystyle\int\limits_{a}^{b}}
%EndExpansion
\sqrt{2-t^{2}}\,dt,\qquad\lbrack a,b]\subset\lbrack-\sqrt{2},\sqrt{2}].
\]
Again, this result is not new and corresponds to the famous semi-circle law
for the eigenvalues of random matrices.

\subsection{Bell polynomials \label{Bell}}

The Bell polynomials satisfy the equation \cite{MR1503161}%
\[
P_{n+1}(x)=xP_{n}^{\prime}(x)+xP_{n}(x).
\]
In \cite{MR2250392}, we obtained asymptotic approximations for these
polynomials using a discrete version of the ray method. Choosing $\phi(n)=n,$
we get%
\begin{equation}
a(z)=z,\quad b(z)=z. \label{abBell}%
\end{equation}
Using (\ref{abBell}) and $\sigma=1$ in (\ref{w-y}), (\ref{E}), (\ref{R}) and
(\ref{x-z}), we have%
\begin{equation}
w(z)=S(z)+1+\frac{1}{z},\quad E(z)=1,\quad R(z)=-\frac{1}{2z},\quad x=\frac
{2}{z}. \label{wBell}%
\end{equation}
Hence, $R(x)=-\frac{1}{4}x$ and the canonical form of the Abel equation
(\ref{canonical}) is%
\[
w\left(  \frac{dw}{dx}-1\right)  =-\frac{1}{4}x,
\]
and therefore%
\begin{equation}
w(x)=\frac{x}{2}\left[  1+\frac{1}{\mathrm{W}\left(  Cx\right)  }\right]
\label{wSol}%
\end{equation}
where $\mathrm{W}\left(  z\right)  $ is the Lambert W function
\cite{MR1414285} defined by%
\begin{equation}
z=\mathrm{W}\left(  z\right)  e^{\mathrm{W}\left(  z\right)  },\quad
\mathrm{W}\left(  0\right)  =0. \label{LW}%
\end{equation}
with \cite[4.13.2]{MR2723248}%
\[
\mathrm{W}\left(  -e^{-1}\right)  =-1,\quad\mathrm{W}\left(  e\right)  =1,
\]
and the differentiation property%
\[
\mathrm{W}^{\prime}\left(  z\right)  =\frac{e^{-\mathrm{W}\left(  z\right)  }%
}{\mathrm{W}\left(  z\right)  +1}.
\]

Using (\ref{wBell}) in (\ref{wSol}), we obtain%

\begin{equation}
S(z)=\left[  z\mathrm{W}\left(  \frac{2C}{z}\right)  \right]  ^{-1}-1.
\label{bellsol}%
\end{equation}
The Taylor series of the function $\mathrm{W}\left(  z\right)  $ around $0$ is
\cite{MR1809988}%
\begin{equation}
\mathrm{W}\left(  z\right)  =\sum\limits_{n=1}^{\infty}\left(  -n\right)
^{n-1}\frac{z^{n}}{n!},\quad\left\vert z\right\vert <e^{-1}. \label{WTaylor}%
\end{equation}
The function defined by this series can be extended to a holomorphic function
defined on all complex numbers with a branch cut along the interval
$(-\infty,-e^{-1}]$; this holomorphic function defines the principal branch of
the Lambert W function. Using (\ref{WTaylor}) in (\ref{bellsol}), we get%
\[
S(z)\sim\frac{1}{2C}-1+\frac{1}{z},\quad z\rightarrow\infty.
\]
Thus, we need $C=\frac{1}{2}$ and conclude that%
\[
S(z)=\left[  z\mathrm{W}\left(  \frac{1}{z}\right)  \right]  ^{-1}%
-1=e^{\mathrm{W}\left(  \frac{1}{z}\right)  }-1.
\]

Applying the Lagrange Inversion Formula \cite{MR2172781} to (\ref{LW}), we
have%
\[
e^{\mathrm{W}\left(  z\right)  }=\sum\limits_{n=0}^{\infty}\left(  1-n\right)
^{n-1}\frac{z^{n}}{n!},\quad\left\vert z\right\vert <e^{-1}.
\]
Hence,%
\[
S(z)=e^{\mathrm{W}\left(  \frac{1}{z}\right)  }-1=\sum\limits_{n=0}^{\infty
}\frac{\left(  -n\right)  ^{n}}{\left(  n+1\right)  !}\frac{1}{z^{n+1}}%
,\quad\left\vert z\right\vert >e.
\]
The function $S(z)$ has a branch cut along the interval $[-e,0].$

From (\ref{Inversion}), we have%
\[
S(z)=%
%TCIMACRO{\dint \limits_{-e}^{0}}%
%BeginExpansion
{\displaystyle\int\limits_{-e}^{0}}
%EndExpansion
\frac{d\psi(t)}{z-t},\quad z\notin\lbrack-e,0],
\]
with%
\[
\psi^{\prime}(t)=\frac{1}{\pi}\operatorname{Im}\exp\left[  \mathrm{W}\left(
\frac{1}{t}\right)  \right]  .
\]
Using the results in \cite{MR2959457}, we also have the integral
representation%
\[
S(z)=\frac{1}{\pi}%
%TCIMACRO{\dint \limits_{0}^{\pi}}%
%BeginExpansion
{\displaystyle\int\limits_{0}^{\pi}}
%EndExpansion
\frac{v^{2}+\left[  1-v\cot\left(  v\right)  \right]  ^{2}}{v^{2}\csc
^{2}\left(  v\right)  \left[  z+v^{-1}\sin\left(  v\right)  e^{v\cot\left(
v\right)  }\right]  }dv.
\]

In \cite{MR1820893}, C. Elbert studied the zero asymptotics of $P_{n}(x)$, and
obtained
\begin{equation}
\psi\left(  t\right)  =1+\frac{1}{\pi}\left[  \operatorname{Im}\frac
{1}{\mathrm{W}\left(  \frac{1}{t}\right)  }-\arg\mathrm{W}\left(  \frac{1}%
{t}\right)  \right]  , \label{psiBell}%
\end{equation}
although he didn't identify the function appearing in his formulas with the
Lambert W function. His method was completely different, and was based on his
previous work on the asymptotic analysis of $P_{n}(x)$ using the saddle point
method \cite{MR1820892}.

In this case, we have
\[
\lim_{n\rightarrow\infty}\frac{1}{n}\#\{\text{zeros of }P_{n}(nx)\text{ in
}[a,b]\}=\frac{1}{\pi}%
%TCIMACRO{\dint \limits_{a}^{b}}%
%BeginExpansion
{\displaystyle\int\limits_{a}^{b}}
%EndExpansion
\operatorname{Im}\exp\left[  \mathrm{W}\left(  \frac{1}{t}\right)  \right]
\,dt,\qquad\lbrack a,b]\subset\lbrack-e,0].
\]
In Figure \ref{bellfig} we plot the zero counting measure $\psi_{75}\left(  z\right)  $
defined in (\ref{psin}) and the measure $\psi\left(  t\right)  $ defined in
(\ref{psiBell}), to illustrate the accuracy of our results.

\begin{figure}[t]
\begin{center}
{\resizebox{17cm}{!}{\includegraphics{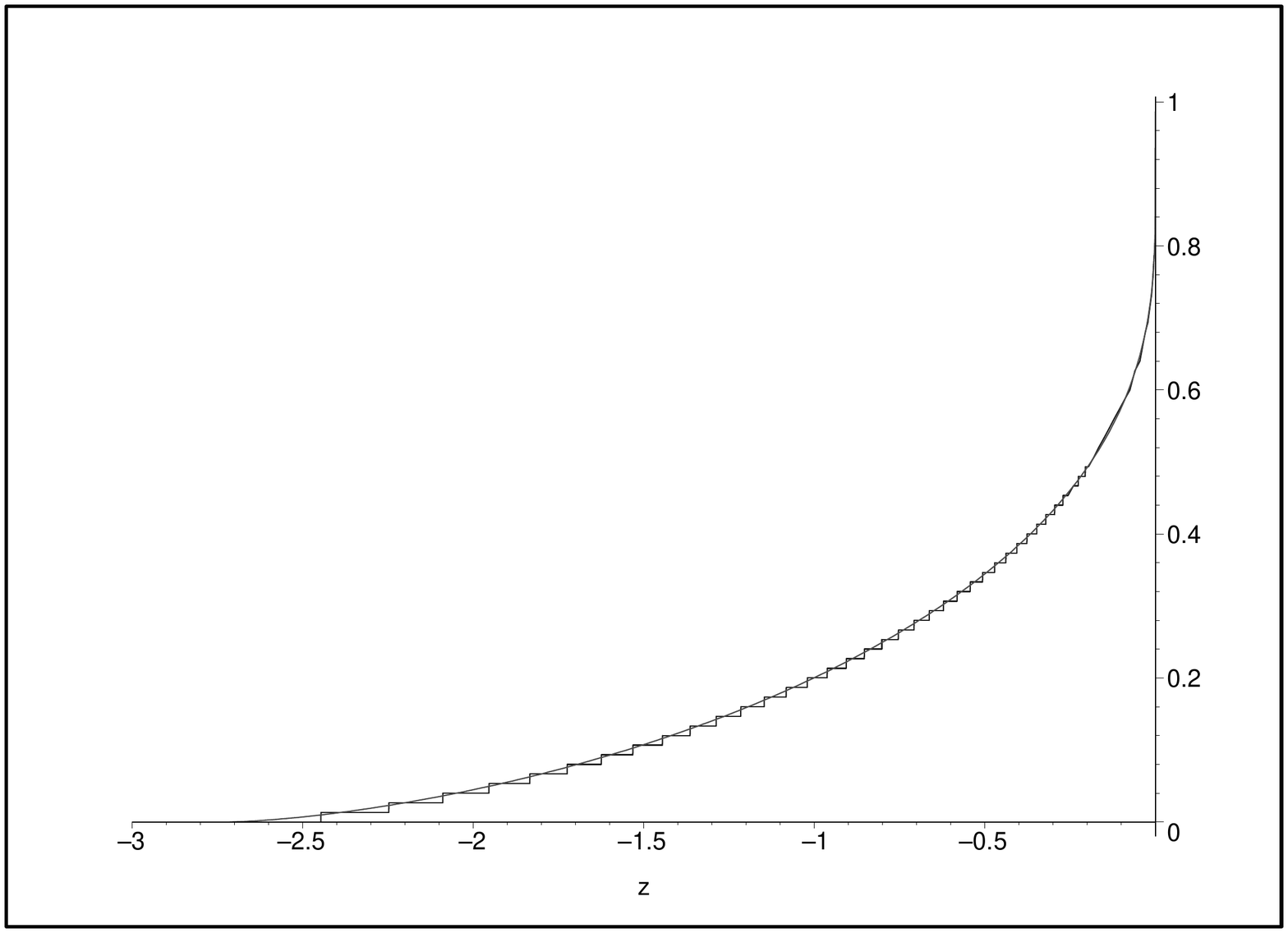}}}
\end{center}
\caption{A plot of $\psi_{75}$ (step function) and $\psi$ (solid line).}%
\label{bellfig}%
\end{figure}

\subsection{Inverse error function polynomials \label{Error}}

Let's consider the polynomials $P_{n}(x)$ defined by $P_{0}(x)=1$ and%
\[
P_{n+1}(x)=P_{n}^{\prime}(x)+\left(  n+1\right)  xP_{n}(x).
\]
The polynomials $P_{n}(x)$ arise in the computation of higher derivatives of
the inverse error function \cite{MR1986919}. Since they have purely imaginary
zeros, we set
\[
P_{n}(x)=n!\left(  \mathrm{i}\right)  ^{n}Q_{n}(-\mathrm{i}x),
\]
and obtain a family of monic polynomials with real zeros, satisfying%
\[
Q_{n+1}(x)=-\frac{1}{n+1}Q_{n}^{\prime}(x)+xQ_{n}(x).
\]
In this case, we can take $\phi(n)=1$ and get%

\begin{equation}
\lim_{n\rightarrow\infty}nA_{n}(x)=-1=a(x),\quad\lim_{n\rightarrow\infty}%
B_{n}(x)=x=b(x). \label{aberror}%
\end{equation}

Using (\ref{aberror}) and $\sigma=0$ in (\ref{solRicat}) gives%
\[
S(z)=z+\exp\left(  \frac{z^{2}}{2}\right)  \left[  \mathrm{i}\sqrt{\frac{\pi
}{2}}\mathrm{\operatorname{erf}}\left(  \frac{\mathrm{i}z}{\sqrt{2}}\right)
-C\right]  ^{-1}%
\]
where $\mathrm{\operatorname{erf}}\left(  \cdot\right)  $ is the error
function defined by \cite[7.2.1]{MR2723248}%
\[
\mathrm{\operatorname{erf}}\left(  z\right)  =\frac{2}{\sqrt{\pi}}%
\int\limits_{0}^{z}\exp\left(  -t^{2}\right)  dt
\]
with asymptotic behavior \cite[7.12.1]{MR2723248}%
\[
\mathrm{\operatorname{erf}}\left(  z\right)  \sim1-\frac{1}{\sqrt{\pi}}%
\exp\left(  -z^{2}\right)  \left(  \frac{1}{z}-\frac{1}{2z^{3}}\right)  ,\quad
z\rightarrow\infty.
\]
Hence, we have%
\[
\mathrm{i}\sqrt{\frac{\pi}{2}}\mathrm{\operatorname{erf}}\left(
\frac{\mathrm{i}z}{\sqrt{2}}\right)  \sim-\exp\left(  \frac{z^{2}}{2}\right)
\left(  \frac{1}{z}+\frac{1}{z^{3}}\right)  ,\quad z\rightarrow\infty.
\]
Along the imaginary axis, we get%
\[
S(\mathrm{i}z)\sim\mathrm{i}z+\exp\left(  -\frac{z^{2}}{2}\right)  \left[
\left(  -\mathrm{i}\sqrt{\frac{\pi}{2}}-C\right)  \right]  ^{-1},\quad
z\rightarrow\infty,
\]
and we conclude that we need to choose
\[
C=-\mathrm{i}\sqrt{\frac{\pi}{2}}.
\]
Thus,%
\[
S(z)=z-\mathrm{i}\sqrt{\frac{2}{\pi}}\exp\left(  \frac{z^{2}}{2}\right)
\left[  1+\mathrm{\operatorname{erf}}\left(  \frac{\mathrm{i}z}{\sqrt{2}%
}\right)  \right]  ^{-1}.
\]
From (\ref{req}), we have%
\[
S(z)=%
%TCIMACRO{\dsum \limits_{n=1}^{\infty}}%
%BeginExpansion
{\displaystyle\sum\limits_{n=1}^{\infty}}
%EndExpansion
\frac{c_{n}}{z^{n}},
\]
where the coefficients $c_{n}$ satisfy the recurrence%
\[
c_{n+2}=nc_{n}+%
%TCIMACRO{\dsum \limits_{k=0}^{n-1}}%
%BeginExpansion
{\displaystyle\sum\limits_{k=0}^{n-1}}
%EndExpansion
c_{k+1}c_{n-k},\quad n=1,2,\ldots,
\]
with $c_{1}=1$ and $c_{2}=0.$

Using the identity \cite[7.5.1]{MR2723248}
\[
\mathrm{\operatorname{erf}}\left(  \mathrm{i}z\right)  =\frac{2\mathrm{i}%
}{\sqrt{\pi}}\exp\left(  z^{2}\right)  \mathrm{daw}\left(  z\right)  ,
\]
where $\mathrm{daw}\left(  \cdot\right)  $ is Dawson's integral defined by
\cite[7.2.5]{MR2723248}%
\[
\mathrm{daw}\left(  z\right)  =\int\limits_{0}^{z}\exp\left(  t^{2}%
-z^{2}\right)  dt,
\]
we can write%
\begin{equation}
S(z)=z-\sqrt{\frac{2}{\pi}}\exp\left(  \frac{z^{2}}{2}\right)  \frac{\frac
{2}{\sqrt{\pi}}\exp\left(  \frac{z^{2}}{2}\right)  \mathrm{daw}\left(
\frac{z}{\sqrt{2}}\right)  +\mathrm{i}}{\frac{4}{\pi}\exp\left(  z^{2}\right)
\mathrm{daw}^{2}\left(  \frac{z}{\sqrt{2}}\right)  +1}. \label{SError}%
\end{equation}

From (\ref{Inversion}), we have%
\[
S(z)=%
%TCIMACRO{\dint \limits_{-\infty}^{\infty}}%
%BeginExpansion
{\displaystyle\int\limits_{-\infty}^{\infty}}
%EndExpansion
\frac{d\psi(t)}{z-t},\quad z\notin\mathbb{R},
\]
with%
\[
\psi^{\prime}(t)=\sqrt{\frac{2}{\pi}}\exp\left(  \frac{t^{2}}{2}\right)
\left[  4\exp\left(  t^{2}\right)  \mathrm{daw}^{2}\left(  \frac{t}{\sqrt{2}%
}\right)  +\pi\right]  ^{-1}.
\]
But since%
\[
\frac{d}{dt}\exp\left(  \frac{t^{2}}{2}\right)  \mathrm{daw}\left(  \frac
{t}{\sqrt{2}}\right)  =\frac{1}{\sqrt{2}}\exp\left(  \frac{t^{2}}{2}\right)
,
\]
we obtain%
\begin{equation}
\psi(t)=\frac{1}{\pi}\arctan\left[  \frac{2}{\sqrt{\pi}}\exp\left(
\frac{t^{2}}{2}\right)  \mathrm{daw}\left(  \frac{t}{\sqrt{2}}\right)
\right]  +\frac{1}{2}. \label{PsiError}%
\end{equation}
We conclude that%
\[
\lim_{n\rightarrow\infty}\frac{1}{n}\#\{\text{zeros of }Q_{n}(x)\text{ in
}[a,b]\}=\psi(b)-\psi(a).
\]
In Figure \ref{errorfig} , we plot the zero counting measure $\psi_{100}\left(  z\right)  $
defined in (\ref{psin}) and the measure $\psi\left(  t\right)  $ defined in
(\ref{PsiError}).

\begin{figure}[t]
\begin{center}
{\resizebox{17cm}{!}{\includegraphics{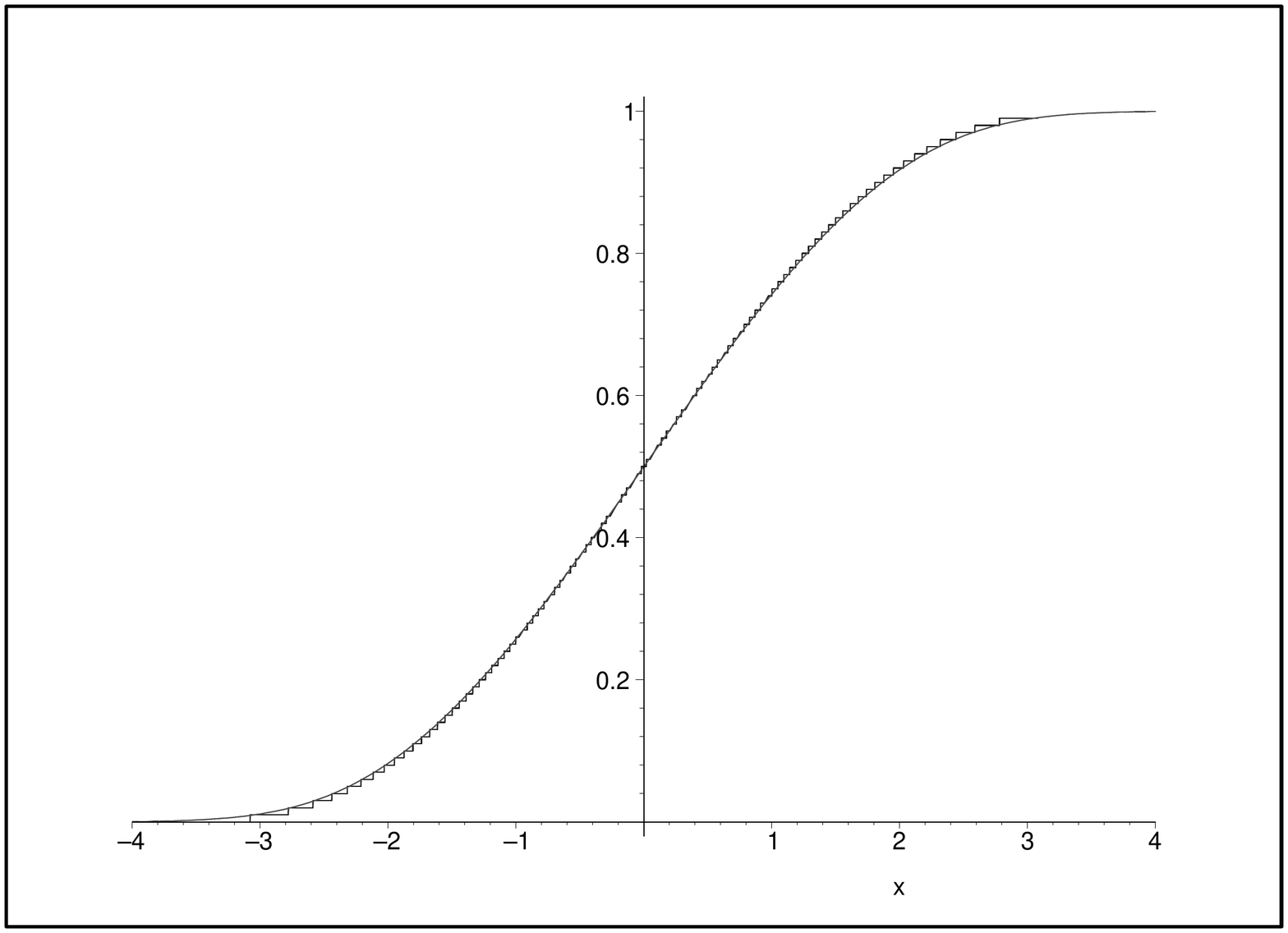}}}
\end{center}
\caption{A plot of $\psi_{100}$ (step function) and $\psi$ (solid line).}%
\label{errorfig}%
\end{figure}

In \cite{MR2966475}, we analyzed the polynomials $P_{n}(x)$ asymptotically
and, among others, we considered the limit $n\rightarrow\infty,$ with $x=y/n$
and $y=O(1).$ We obtained the asymptotic approximation%
\[
P_{n}(x)\sim n^{n}e^{-n}\sqrt{\frac{2n}{\ln(n)}}\left(  \sqrt{\frac{2}{\pi}%
}\right)  ^{n}\left[  \exp\left(  \sqrt{\frac{2}{\pi}}y\right)  +\left(
-1\right)  ^{n}\exp\left(  -\sqrt{\frac{2}{\pi}}y\right)  \right]  ,
\]
and therefore%
\[
Q_{n}(x)\sim\mathrm{i}^{n}\frac{1}{\sqrt{\pi\ln(n)}}\left(  \sqrt{\frac{2}%
{\pi}}\right)  ^{n}\left[  \exp\left(  -\sqrt{\frac{2}{\pi}}nx\mathrm{i}%
\right)  +\left(  -1\right)  ^{n}\exp\left(  \sqrt{\frac{2}{\pi}}%
nx\mathrm{i}\right)  \right]  ,
\]
where we have used Stirling's formula \cite[5.11.7]{MR2723248}%
\[
n!\sim\sqrt{2\pi n}n^{n}e^{-n},\quad n\rightarrow\infty.
\]
Thus, we get%
\[
Q_{n}(x)\sim\left\{
\begin{array}
[c]{c}%
2\left(  -\frac{2}{\pi}\right)  ^{\frac{n}{2}}\left[  \pi\ln(n)\right]
^{-\frac{1}{2}}\cos\left(  \sqrt{\frac{2}{\pi}}nx\right)  ,\quad n\text{ \ is
even}\\
-\sqrt{2}\left(  -\frac{2}{\pi}\right)  ^{\frac{n+1}{2}}\left[  \pi
\ln(n)\right]  ^{-\frac{1}{2}}\sin\left(  \sqrt{\frac{2}{\pi}}nx\right)
,\quad n\text{ \ is odd}%
\end{array}
\right.  .
\]
It follows that $x_{k,n}$ (the zeros of $Q_{n}),$ are approximated
asymptotically by%
\[
x_{k,n}\sim\left\{
\begin{array}
[c]{c}%
\frac{\pi^{\frac{3}{2}}}{\sqrt{2}}\frac{k}{n}+\left(  \frac{\pi}{2}\right)
^{\frac{3}{2}}\frac{1}{n},\quad n\text{ \ is even}\\
\frac{\pi^{\frac{3}{2}}}{\sqrt{2}}\frac{k}{n},\quad n\text{ \ is odd}%
\end{array}
\right.  ,
\]
with $k=0,\pm1,\pm2,\ldots.$ Hence, the zero counting measure (\ref{psin}) can
be approximated by%
\begin{equation}
d\psi_{n}\left(  t\right)  \sim\frac{\sqrt{2}}{\pi^{\frac{3}{2}}}dt,\quad
n\rightarrow\infty. \label{psinError}%
\end{equation}
But from (\ref{PsiError}), we have%
\[
\psi(t)\sim\frac{1}{2}+\frac{\sqrt{2}}{\pi^{\frac{3}{2}}}t,\quad
t\rightarrow0,
\]
in agreement with (\ref{psinError}).

\section{Conclusion}

In this paper we have investigated the asymptotic zero distribution of a
family of polynomials satisfying a differential-difference equation of the
form
\[
P_{n+1}(x)=A_{n}(x)P_{n}^{\prime}(x)+B_{n}(x)P_{n}(x), \qquad n\geq0,
\]
where $A_{n}$ are polynomials of degree at most $2$ and $B_{n}$ are
polynomials of degree at most $1$. We have shown that, assuming the zeros of
the polynomials interlace and after appropriate scaling using some regularly
varying function $\phi(n)$, the Stieltjes transform of the asymptotic zero
distribution satisfies a differential equation of Riccati or Abel type, which
can be solved explicitly. We have illustrated this result for the classical
orthogonal polynomials of Jacobi, Laguerre, and Hermite, for which the
asymptotic zero distribution is already well known, but also for two families
of polynomials which are not orthogonal polynomials: the Bell polynomials and
polynomials related to the inverse error function. One of the main ingredients
in this paper is Theorem \ref{thm2} which shows that the sequence of zero
counting measures with regularly varying scaling is relatively compact in the
Skorohod metric on $D[0,1]$.

\end{document}